\documentclass{AIMS}
\usepackage{amsmath}
\usepackage{paralist}
\usepackage{graphics}  
\usepackage[colorlinks=true]{hyperref}  
\hypersetup{urlcolor=blue, citecolor=red}

\def\currentvolume{X}
 \def\currentissue{X}
  \def\currentyear{200X}
   \def\currentmonth{XX}
    \def\ppages{X--XX}

\newtheorem{theorem}{Theorem}[section]

\newtheorem{lemma}[theorem]{Lemma}

\theoremstyle{definition}

\newtheorem{remark}{Remark}[section]

\numberwithin{equation}{section}

\newcommand{\Rn}{\mathbb{R}^n}
\newcommand{\vp}{v}
\newcommand{\bfR}{\mathbb{R}}

\title[Instability of steady states]{Nonlinear instability of solutions in parabolic and hyperbolic diffusion}

\author[Stephen Pankavich and Petronela Radu]{}

\subjclass{Primary: 35B35, 35B05, 35B30; Secondary: 35L70, 35K55.}
 \keywords{evolution equations, sign-changing damping, instability, variable coefficients, steady states.}

 \email{pankavic@mines.edu}
 \email{pradu@math.unl.edu}

\thanks{The first author was supported in part by NSF grants DMS-0908413 and DMS-1211667.  The second author was supported in part by NSF grant DMS-0908435}

\begin{document}

\maketitle

\centerline{\scshape Stephen Pankavich}
\medskip
{\footnotesize
% please put the address of the first author
 \centerline{Department of Mathematics}
 \centerline{United States Naval Academy}
 \centerline{Annapolis, MD 21402} \ \\ 
 \centerline{Department of Applied Mathematics and Statistics}
 \centerline{Colorado School of Mines}
 \centerline{Golden, CO 80002}

} % Do not forget to end the {\footnotesize by the sign }

\medskip

\centerline{\scshape Petronela Radu}
\medskip
{\footnotesize
 % please put the address of the second  and third author
 \centerline{Department of Mathematics}
 \centerline{University of Nebraska, Lincoln,}
  \centerline{Lincoln, NE 68588}

}

\bigskip

% The name of the associate editor will be entered by an editorial staff
% "Communicated by the associate editor name" is not needed for special issue.
 \centerline{(Communicated by the associate editor name)}

\begin{abstract}
We consider semilinear evolution equations of the form $a(t)\partial_{tt}u + b(t) \partial_t u + Lu = f(x,u)$ and $b(t) \partial_t u + Lu = f(x,u),$ with possibly unbounded $a(t)$ and possibly sign-changing damping coefficient $b(t)$, and determine precise conditions for which linear instability of the steady state solutions implies nonlinear instability. More specifically, we prove that linear instability with an eigenfunction of fixed sign gives rise to nonlinear instability by either exponential growth or finite-time blow-up.  We then discuss a few examples to which our main theorem is immediately applicable, including evolution equations with supercritical and exponential nonlinearities.
\end{abstract}

\section{Introduction}

We consider the second order evolution equation
\begin{equation} \label{one}
a(t)\partial_{tt}u + b(t) \partial_t u + Lu = f(x,u), \qquad x \in \Omega, t > 0
\end{equation}
and its first-order counterpart
\begin{equation} \label{paraone}
b(t) \partial_t u + Lu = f(x,u), \qquad x \in \Omega, t > 0
\end{equation}
where $L$ is a linear differential operator with smooth, bounded coefficients, $f$ is a nonlinear source, $b(t)$ is a damping term, and $a(t)$ is a time dependent coefficient related to the relaxation time of the system (\ref{one}). We consider either $\Omega = \Rn$, or $\Omega \subset \Rn$ is bounded with smooth boundary.  In the case of a bounded spatial domain, we impose Dirichlet boundary conditions on functions in the domain of $L$ and their derivatives up to a suitable order.
An important question in the study of (\ref{one}) and (\ref{paraone}) is to understand the qualitative behavior of special types of solutions.  In the present context, we consider steady-state solutions $\vp \in C^2(\Omega)$ that satisfy
$$L\vp = f(x,\vp)$$
and determine specific conditions under which they are {\it nonlinearly} unstable if it has already been determined that they are {\it linearly} unstable.  

\subsection{Physical applications} If $L$ is an elliptic operator, then (\ref{one}) can be viewed as a model of hyperbolic diffusion with (\ref{paraone}) being its parabolic equivalent. There is a strong connection between the two models at the mathematical level (\cite{M,N,RTY2,RTY3}), as well as at the physical level since both have been proposed to model diffusion of heat or mass. Hyperbolic diffusion models such as (\ref{one}) have been proposed as a solution to eliminate the ``conduction paradox" which renders an unbounded propagation speed wave when the initial signal is perturbed (see \cite{RB} and the references within). This paradox is related to the infinite speed of propagation that is specific to parabolic models. To eliminate this unrealistic feature, one often replaces Fourier's (or Fick's) flux law with the Cattaneo-Vernotte law for the flux $q$ 
 \begin{equation}\label{CV}
 q+\tau_0 (t) \frac{\partial q}{\partial t}=-k\nabla \theta.
 \end{equation}
 Here $\theta$ is the temperature, $\tau_0(t)$ is the thermal relaxation that may be dependent on time, and $k$ is the thermal conductivity of the material. The above law was introduced in \cite{C,V} to describe unsteady heat conduction in phenomena such as the second sound of helium. Combining the standard conservation law 
 \[
 \theta_t + \text{div } q =0
 \]
 with the Cattaneo-Vernotte flux law yields 
   \[
 \theta_t-\text{div}(k \nabla \theta)+\tau_0(t) \theta_{tt}=0.
 \]
This is a generalization of the Maxwell-Cattaneo equation with the general description given by  
 \begin{equation}\label{MC}
a(t)u_{tt} + b u_t -\text{div}(k\nabla u) = 0.
\end{equation}
The connection between the two models will be clearly illustrated in our results, as we will show that the same mechanism applies to both the parabolic and hyperbolic settings when studying the instability of steady states. 

From the point of view of dynamics in elasticity, there has been a long interest in studying nonlinear, second-order evolution equations in the presence of {\it positive} damping. This interest is motivated by the need to determine the state of a physical system under an energy decreasing force, such as friction. However, systems with sign-changing damping are also important in applications as they appear in Aerodynamics, e.g.  the nose wheel shimmy of an airplane on which a hydraulic shimmy damper acts \cite{S}; Mesodynamics, as within a laser driven pendulum \cite{deS}; Quantum Field Theory, e.g. the Landau instability of Bose condensates \cite{KN}; and the macroscopic world, including a well-known model of suspension bridges \cite{LM}.  We refer the reader to \cite{FM} for more references arising in the physics and engineering literature. Our work applies to both classes of models, including positive as well as negative damping.

\subsection{Mathematical Applications}

The connection between the linear and nonlinear stability properties of evolution equations is a problem of long-standing interest within the study of partial differential equations.  In recent years, quite a bit of attention has been devoted to understanding this issue for semilinear evolution equations, with many articles devoted to second-order parabolic and hyperbolic equations. In particular, the stability and instability of steady solutions of nonlinear parabolic and wave equations has been investigated by a variety of authors including Grillakis, Shatah, and Strauss \cite{GSS}, Gui, Ni, and Wang \cite{GNW}, Souplet and Zhang \cite{SZ}, and Strauss and Karageorgis \cite{KS}, among many others.  
%For hyperbolic PDEs, it is also in general difficult to establish rigorous criteria for stability; for instance, it is well known that there are examples where linear instabilities are not linked to the spectrum of the linearized operator \cite{}. On the other hand, such results are not applicable to hyperbolic PDEs in which the principal term is nonlinear. 
Over the past few years, additional interest has arisen in understanding such  properties for evolution equations involving higher-order elliptic operators, such as the biharmonic operator $\Delta^2$ (see \cite{Kara} and references therein).  Similarly, the addition of variable coefficients is often problematic when applying previous methods. Though some techniques have been adapted from the second-order case to study higher-order operators, the majority of methods do not apply.  Additionally, the subtle effect of a change in domain or dimension of the problem can largely influence the associated stability properties of solutions to these equations.  In the current paper, we present results concerning the instability of steady solutions to evolution equations that involve very general conditions on the spatial operator $L$, the associated variable coefficients $a(t)$ and $b(t)$, the nonlinearity $f$, and the spatial domain $\Omega$.  Our assumptions allow for the application of these theorems to a number of specific cases, including $L = -\Delta$, $L = \sum_{i,j=1}^N \frac{\partial}{\partial x_i}  \left (a_{ij}(x) \frac{\partial}{\partial x_j} \right )$, and possibly $L = \Delta^m$ ($m \in \mathbb{N}$)  posed on the whole space $\Omega = \Rn$ or on bounded domains, with convex or concave nonlinearities, and coefficients that reside in suitable $L^p$ spaces. Our results apply to a wide class of initial data, including those for which the initial perturbations are assumed to be nonnegative.

In order to prove the main theorems, we will rely on a variant of Kaplan's eigenfunction method \cite{K} which was originally used to study the growth of solutions to quasilinear parabolic equations.  Other variants of this particular method have been used by several authors \cite{G, Kara, YZ} to study existence and uniqueness problems. More recently, this method was also utilized by Strauss and Karageorgis \cite{KS} to prove theorems regarding the instability of steady states for a family of nonlinear parabolic and hyperbolic equations.  For first-order evolution equations, Shatah and Strauss \cite{SS} have proved the extension of linearly unstable solutions to nonlinear instability under very general conditions. In the present study, we focus on arriving at similar results in the presence of time-dependent, and possibly sign-changing, coefficients for problems involving more general operators and posed on a variety of spatial domains.  One impetus for considering such coefficients arises from recent papers \cite{FM, FM1} in which it is shown that steady states of nonlinear wave equations with sign-changing damping coefficients are \emph{stable} under certain assumptions.  In contrast, we will prove that the analogous solutions with damping terms that allow for coefficients to change sign are \emph{unstable}.

To establish our instability result for second-order evolution equations, we show that the $L^2$ norm of the perturbed solution must grow exponentially in time. Notice that this is a much stronger result than just instability itself, as the statement implies that the solution exists in any given neighborhood of $\vp$, that it never returns to the region, and that it exits this neighborhood exponentially fast. For such problems, the $L^2$ norm lower bounds the natural energy norm for the system, and hence instability in $L^2$ implies unstable behavior of the energy.  Under stricter assumptions on the nonlinearity and eigenfunction, the instability can be shown to occur by blow-up, so that the $L^2$ norm of the perturbed solution must tend to infinity at some finite time. In the case of the first-order equation (\ref{paraone}), we make the additional assumption that the damping is necessarily positive, $b(t) > 0$ for all $t \in [0,\infty)$. This is imposed in order to ensure that the problem remains strictly parabolic and avoid any difficulty stemming from the widely-known ill-posedness of the backward diffusion equation.  With this condition, we prove instability in the $L^\infty$ norm by exponential growth or by blow-up.

In the special case that $a(t)\equiv 1$ and $b(t) = b$ is constant, $\Omega = \Rn$, and $L$ is second-order, similar instability questions have been previously studied.  More specifically, the instability (with a negative principal eigenvalue) of the steady state $\vp\in C^2$ of the linearized problems 
$$\partial_{tt}u + b \partial_t u + Lu  = \partial_uf(x,\vp)u, \qquad x \in \Rn,\, t > 0,$$ 
$$ \partial_t u + Lu = \partial_uf(x,\vp)u, \qquad x \in \Rn, t\, > 0 $$
was shown in \cite{KS} to imply the instability of $\vp$ as a solution to the corresponding nonlinear problems 
$$\partial_{tt}u + b \partial_t u + Lu = f(x,u), \qquad x \in \Rn, \,t > 0$$
$$\partial_t u + Lu = f(x,u), \qquad x \in \Rn,\, t > 0$$ for $f$ convex, but subject to a quadratic condition involving both the value of $b$ and the initial size of the perturbation under consideration.  

We point out that our results were inspired by the work of \cite{KS}, however the proofs of \cite{KS} cannot be extended to the variable coefficient case. To overcome this difficulty, we use continuity arguments that allow time-dependent coefficients for the damping, as well as for the second derivative term. Thus, while the main theorems of \cite{KS} required a condition on the constant $b$, our main theorems hold for any choice of time-dependent damping coefficient $b(t)$ which lies in a suitable $L^p$ space and $a(t)$ positive and bounded below. Since $b(t)$ is not necessarily of a fixed sign, one cannot utilize standard energy estimates to derive stability or instability results.  Additionally, the inclusion of variable coefficients is a non-trivial matter - one that cannot be settled by a simple change of variables, as no such transformation will convert this system into one with constant coefficients. In fact, the damping term $b(t) \partial_t u$ can have a significant impact on the behavior of solutions, even for the linear analogues of (\ref{one}) and (\ref{paraone}).  It is well known \cite{M, RTY2} that in the presence of damping, the asymptotic profile of solutions to the linear wave equation is well-behaved and resembles the Gaussian profile of solutions to the linear diffusion equation. Moreover, the difference between solutions of the wave and diffusion equations tends to zero as $t\to \infty$ faster than the decay of either solution on its own. This remarkable behavior is called the strong {\it diffusion phenomenon} and was shown to hold for a variety of systems \cite{HM, HL,N, RTY3}.  With the addition of the {\it time-dependent} coefficients in the nonlinear problem, the question arises as to whether steady states become unstable under smooth perturbations for more general evolution equations, and the present study is devoted to addressing this open question.

Regarding the growth of the damping coefficient for the second-order equation, note that a prototype that satisfies the assumptions outlined in our main theorems is 
 $$b(t) =b_0(1+t)^{-\alpha}, \quad \alpha >0.$$ 
On the other hand, we note that for coefficients $b(t)$ that behave like
\begin{equation}\label{agrw}
b(t)\sim (1+t)^{-\alpha}, \quad \alpha \in (0,1)
\end{equation} we know (see \cite{W} for example) that the total energy of the linearly damped system
\begin{equation}
\partial_{tt}u + b(t) \partial_t u -\Delta u = 0, \qquad x \in \Omega,\, t > 0
\end{equation}
decays to zero as $t\to \infty$. Since our work applies for the range of exponents $\alpha\in(0,1)$, we are able to identify conditions on the source term for which the inhomogeneous system becomes unstable, though its homogeneous counterpart remains stable.

More interestingly, perhaps, our result for the second-order evolution equation allows coefficients that may change sign in time. To our knowledge these are the first {\it instability} results in this direction, whereas the first stability results for sign-changing systems were reported recently for a nonlinear problem in \cite{FM} (based on the earlier work \cite{FM1}). The authors show in \cite{FM,FM1} that if the damping $b(t)$ is negative for a sufficiently small length of time, then steady solutions remain \emph{stable} for the system
\begin{equation}\label{dwe}
\partial_{tt}u + b(t) \partial_t u -\Delta u = f(u), \qquad x \in \Omega, t > 0
\end{equation}
where $\Omega \subset \Omega$ is bounded and $f(u)$ is absorbing with subcritical growth (i.e. $f(u) = - u|u|^{p-1}, 1\leq p \leq \frac{n}{n-2}$). In contrast, we will show in the applications section that for accretive forcing terms $f(u)=u|u|^{p-1}$ the system (\ref{dwe}) exhibits instability of steady states, with exponential blow-up, for all exponents $p$.

The paper is structured as follows.  In the next section, we prove the main instability theorems for both the first and second-order evolution equations.  Section $3$ then contains several lemmas used to prove Theorems \ref{T1} -- \ref{T2}. Finally, in Section $4$ we discuss a few examples of well-known problems to which our primary results are immediately applicable. Throughout, the quantity $C$ represents a generic constant that may change in value from line to line.

\section{Main results}

In this section, we prove the main nonlinear instability results regarding solutions of the steady equation $$L\vp = f(x,\vp) \qquad x \in \Omega.$$  We consider the initial value problem for the previously described second-order evolution equation with $x \in \Omega$ and $t > 0$:

\begin{equation} \label{two}
\left. \begin{gathered}
a(t)\partial_{tt}u + b(t) \partial_t u + Lu = f(x,u)\\
u(0,x) = \vp(x) + u_0(x)\\
\partial_t u(0,x) = u_1(x). \\
\end{gathered} \right \}
\end{equation}

The steady state $\vp$ is an exact solution of (\ref{two}) if $u_0 = u_1 \equiv 0$.  Thus, we are interested in the behavior of $u$ whenever $u_0$ and $u_1$ are small in some sense. Our results are subtly different for $b \in L^\infty$ and $b \in L^1$.  In short, we find instability for general damping coefficients $b \in L^1_{loc}$, instability by exponential growth for integrable coefficients $b \in L^1$, and instability by either exponential growth or blow-up for essentially bounded coefficients $b \in L^\infty$.

\begin{theorem}
\label{T1}% (2nd-order Equation)
Assume the following:
\begin{enumerate}
\item The adjoint linearized operator $L^* - \partial_u f(x,\vp)$ possesses a negative eigenvalue $\lambda_1 = -\sigma^2$ and corresponding nonnegative eigenfunction $\phi_1 \in L^2(\Omega)$
\item The damping term satisfies $b \in L^1_{loc}(0,\infty)$
\item The coefficient $(a \in C^1(0,T)$ and there exists $a_0>0$ such that $a(t) \geq a_0$ for all $t \in [0,\infty)$.
\item The nonlinearity $f(x,u)$ is $C^1$ and convex in $u$.
\end{enumerate}
Let $T \in (0,\infty]$ be given and assume the initial data $(u_0,u_1) \in H^1(\Omega) \times L^2(\Omega)$ satisfies $$\int_{\Omega} \phi_1(x) u_1(x) \ dx, \int_{\Omega} \phi_1(x) u_0(x) \ dx > 0.$$  If $u$ is a solution of (\ref{two}) such that $u - \vp \in C^1([0,T); L^2(\Omega))$, then we have the following:

\renewcommand{\labelenumi}{\alph{enumi}.}
\begin{enumerate}
\item There is $C_0 > 0$ such that for all $t \in (0,T)$, $$\Vert u(t) - \vp \Vert_2 > C_0\int_{\Omega} \phi_1(x) u_0(x) \ dx > 0.$$ 
\item Assume $ b \in L^\infty(0,T)$ and $a'(t)>0$. If the initial data also satisfies 
$$\int_{\Omega} \phi_1(x) u_1(x) \ dx \geq \left (-B + \sqrt{+\frac{\sigma^2}{a(0)}+B^2}  \right ) \int_{\Omega} \phi_1(x) u_0(x) \ dx$$ 
with $B:=\frac{ \Vert b \Vert_\infty}{a_0},$ then there exist $T_0, C_1>0$ and a positive, increasing function $\mu(t)$ such that for all $t \in (0,T)$ $$\Vert u(t) - \vp \Vert_2 \geq C_1e^{\mu(t)}.$$
The function $\mu$ is given by $\mu(t)=\int_0^t (-B+\sqrt{\frac{\sigma^2}{a(s)}+B^2})ds$ which satisfies $\mu(s)> C\int_0^t\frac{1}{a(s)}ds,$ for some $C>0$. 
\item If $ a,b \in L^\infty(0,T)$, then there exist $T_0, C_1,\mu >0$  such that for all $t \in (T_0,T)$ $$\Vert u(t) - \vp \Vert_2 \geq C_1e^{\mu t}.$$
\item If $ b \in L^1(0,\infty)$, and $ a\in L^{\infty}(0,\infty)$ then there exist $T_0, C_2, \mu > 0$ such that for all $t \in (T_0,T)$ $$\Vert u(t) - \vp \Vert_2 \geq C_2e^{\mu t}.$$
\end{enumerate}
\renewcommand{\labelenumi}{\arabic{enumi}.}
\end{theorem}

\begin{proof} Let $w(t,x) = u(t,x) - \vp(x)$ and define the function $$W(t) = \int_{\Omega} \phi_1(x) w(t,x) \ dx.$$ Using Cauchy-Schwarz, we have $$\vert W(t) \vert \leq \Vert \phi_1 \Vert_{L^2(\Omega)} \cdot \Vert w(t) \Vert_{L^2(\Omega)} = C \Vert w(t) \Vert_{L^2(\Omega)}. $$  Therefore, $W$ is well-defined as long as the $L^2$ norm remains finite and serves as a lower bound for $\Vert w(t) \Vert_2$.  Hence, we may focus on obtaining the necessary growth of $W$ to prove the desired result.

Using the linearity of the left side of the equation in (\ref{two}), we find that $w$ satisfies
\begin{equation}\label{diffsol}
 a(t)\partial_{tt} w + b(t) \partial_t w + Lw = f(x,w + \vp) - f(x,\vp)
 \end{equation}
in the sense of distributions.
Due to the convexity of the nonlinearity, we have $$ a(t)\partial_{tt} w + b(t) \partial_t w + \biggl ( L - \partial_u f(x,\vp) \biggr ) w \geq 0.$$ 
Since $\phi_1$ is nonnegative, we multiply the inequality by $\phi_1(x)$, integrate over $\Omega$, and integrate by parts to obtain
\begin{equation}
\label{Ineq1}
\begin{gathered} 
\int_{\Omega} \phi_1(x) a(t) \partial_{tt} w(t,x) \ dx + \int_{\Omega}\phi_1(x) b(t) \partial_t w(t,x) \ dx\\ +\int_{\Omega} w(t,x) \biggl [ L^* - \partial_u f(x,\vp) \biggr] \phi_1(x)   \ dx \geq 0. 
\end{gathered}
\end{equation}
The assumption on the adjoint operator implies $$ \biggl [ L^* - \partial_u f(x,\vp) \biggr] \phi_1 = - \sigma^2 \phi_1$$
and hence
\begin{equation} \label{three} \begin{gathered} 
a(t)\partial_{tt} \int_{\Omega} \phi_1(x) w(t,x) \ dx +  b(t) \partial_t  \int_{\Omega}\phi_1(x) w(t,x) \ dx\\ - \sigma^2 \int_{\Omega} w(t,x) \phi_1(x)   \ dx \geq 0. \end{gathered} \end{equation}
Because $\partial_t w$ is continuous, we see that $$W'(t) = \int_{\Omega} \phi_1(x) \partial_t w(t,x) \ dx$$ is continuous.  Thus, the differential inequality (\ref{three}) simplifies to become
\begin{equation}
\label{W}
a(t)W''(t)  + b(t) W'(t) - \sigma^2 W(t)  \geq 0
\end{equation}
in the sense of distributions. Since $a(t)\geq a_0>0$ the conclusions of the theorem follow from the lemmas of Section $3$.  More specifically, our assumption on the initial data yields $W(0), W'(0) > 0$.  Therefore, the first result of Lemma \ref{L1} proves
\begin{equation}\label{Wpos}
W(t) > W(0)
\end{equation}
for all $t  \in (0,T)$ and hence the conclusion of part (a).  Similarly, conclusions (b), (c), and (d) follow from the results of Lemmas \ref{L2var},  \ref{L2}, and \ref{L3}, respectively, regarding the behavior of functions $W(t)$ that satisfy (\ref{W}).
\end{proof}

We note that the convexity assumption on the nonlinearity was used in the proof to ensure positivity of the linear portion of the PDE.  If instead, the nonlinearity is concave in $u$, we may arrive at an analogous result with the necessary alterations on the initial data.
\begin{theorem}
\label{T1concave}
Assume the hypotheses of Theorem \ref{T1} hold, with the exceptions that $f(x,u)$ is concave in $u$ rather than convex, and the initial data satisfy $$\int_{\Omega} \phi_1(x) u_1(x) \ dx, \int_{\Omega} \phi_1(x) u_0(x) \ dx < 0.$$
Then, the conclusions of Theorem \ref{T1} remain valid.

\end{theorem}
\begin{proof}
Upon utilizing the concavity assumption to arrive at 
$$ a(t)\partial_{tt} w + b(t) \partial_t w + \biggl ( L - \partial_u f(x,\vp) \biggr ) w \leq 0,$$ 
the proof of the previous theorem is only altered by multiplying this by the negative eigenfunction $-\phi_1(x)$, thereby restoring the correct direction of the inequality
$$\begin{gathered} - \int_{\Omega} \phi_1(x) a(t)\partial_{tt} w(t,x) \ dx - \int_{\Omega} \phi_1(x) b(t) \partial_t w(t,x) \ dx\\ - \int_{\Omega} w(t,x) \biggl [ L^* - \partial_u f(x,\vp) \biggr] \phi_1(x)   \ dx \geq 0 \end{gathered} $$
as in (\ref{Ineq1}).  The behavior of the $L^2$ norm is then controlled by utilizing the quantity $W(t) = - \int_{\Omega} \phi_1(x) w(t,x) \ dx$ and we again arrive at 
$$a(t)W''(t)  + b(t) W'(t) - \sigma^2 W(t)  \geq 0$$
with initial data of the correct sign. The remainder of the proof again relies on the lemmas of the following section and is otherwise unchanged.
\end{proof}

Additionally, if we make further assumptions in the case of a convex nonlinearity, Theorem \ref{T1} can be extended to prove instability due to blow-up.

\begin{theorem}
\label{T1a}% (2nd-order Equation)
Let the hypotheses of Theorem (\ref{T1}) be valid and assume in addition 
\begin{enumerate}
\item The functions $f(x, \vp)$ and $\partial_u f(x, \vp)$ are bounded
\item There are $C > 0$ and $p > 1$ such that $f(x,u) \geq C \vert u\vert^p$ for every $(x,u) \in \Omega \times \mathbb{R}$
\item The eigenfunction and its product with the steady state are integrable, i.e. $ \phi_1, \vp \phi_1 \in L^1(\Omega).$
\end{enumerate}
Let $T \in (0,\infty]$ be given, and assume the coefficients $a(t)$ and $b(t)$ satisfy 
\[
0<a_0\leq a(t)\leq a_1 (t+1)^{r}, r\in [0,1), \quad b \in L^\infty(0,T). 
\]
If $u$ is a solution of (\ref{two}) such that $u - \vp \in C^1([0,T); L^2(\Omega))$, then there exists $T^* < \infty$ such that $$\lim_{t \uparrow T^*} \Vert u(t) - \vp \Vert_2 = +\infty.$$
\end{theorem}

\begin{proof}
To prove the instability by blow-up, we follow \cite{KS} with modifications to adjust for the time-dependence in the damping term $b(t)$. First assume that $T = \infty$. Utilizing the additional assumptions and proceeding as before, the equation (\ref{two}) yields
\begin{eqnarray*}
a(t)\partial_{tt} w + b(t) \partial_t w + \left ( L - \partial_u f(x, \vp) \right ) w & = & f(x,w+  \vp) - f(x, \vp) - \partial_u f(x, \vp) w \\
& \geq & C \vert w +  \vp \vert^p - C (1 + \vert w \vert)
\end{eqnarray*}
As before, we multiply the inequality by the eigenfunction $\phi_1$, integrate, and use the assumption on the adjoint operator to find
\begin{equation}\label{Wineq}
a(t)W''(t) + b(t) W'(t) - \sigma^2 W(t) \geq C \int_{\Omega} \phi_1 \vert w +  \vp \vert^p \ dx - C \int_{\Omega} \phi_1 (1 + \vert w \vert) \ dx.
\end{equation}
Since $\phi_1, \phi_1 \vp \in L^1(\Omega)$, the last term in the inequality satisfies
\begin{eqnarray*} \int_{\Omega} \phi_1 (1 + \vert w \vert) \ dx & \leq & \int_{\Omega} \phi_1 \ dx + \int_{\Omega} \phi_1  \vp \ dx  + \int_{\Omega} \phi_1 \vert w +  \vp \vert\ dx\\
& \leq &   C \left ( 1 + \int_{\Omega} \phi_1 \vert w +  \vp \vert \ dx \right ) .
\end{eqnarray*}
Put $\alpha(t) = \int_{\Omega} \phi_1 \vert w +  \vp \vert^p \ dx$ and $\beta(t) =  \int_{\Omega} \phi_1 \vert w +  \vp \vert \ dx$.  Then, since we know from (\ref{Wpos}) that $W$ stays positive and $a(t)$ is positive, the differential inequality (\ref{Wineq}) implies
\begin{equation}
\label{diffineq}
a(t)W''(t) + b(t) W'(t) \geq C \left ( \alpha(t) - \beta(t) \right )
\end{equation}
for $t$ chosen large enough.  Using Holder's inequality we see that $\alpha(t)$ dominates $\beta(t)$ for large $t$, since for any $p \in (1,\infty)$
$$\beta(t) \leq \left ( \int_{\Omega} \phi_1 \ dx \right )^{\frac{p-1}{p}} \left ( \int_{\Omega} \phi_1 \vert w + \vp \vert^p \ dx \right )^\frac{1}{p} \leq C \alpha(t)^\frac{1}{p}$$ or $$\alpha(t) \geq C \beta(t)^p.$$  Additionally, notice that
$$ W(t) = \int_{\Omega} \phi_1 w \ dx \leq \beta(t) + \int_{\Omega} \phi_1  \vp \ dx \leq \beta(t) + C.$$ By the conclusion of part (b) of the theorem, $W$ grows exponentially fast. Therefore, for $t$ large we find that $\beta(t)$ must also grow exponentially fast, as it is effectively lower bounded by $W$. Using these facts in (\ref{diffineq}), we find 
 \begin{equation}\label{HP}
a(t) W''(t) + b(t) W'(t) \geq C\alpha(t) \geq C \beta(t)^p \geq C W(t)^p.
 \end{equation}
Since $b \in L^\infty(0,T)$ it follows that $b \in L_{loc}^1(0,T)$.  Hence, by conclusion (a), we know $W'$ is positive.  Dividing by $a(t)$ and using the bounds on the coefficients we have
 \[
 W''(t)+\frac{\Vert b \Vert_\infty }{a_0}W'(t)\geq W''(t)+\frac{b(t)}{a(t)}W'(t) \geq \frac{C}{a(t)} W(t)^p \geq C (t+1)^{-r} W(t)^p
 \]
for $t$ large enough. From Lemma \ref{L4}, we arrive at the desired conclusion.
\end{proof}

%\begin{remark}
%\label{R1}
%The assumptions on the initial data can be lessened in the case of $a \in L^\infty(0,T)$.  More specifically, the conclusion of the theorem remains valid under the assumption
%$$\int_{\Omega} \phi_1(x) u_1(x) \ dx + \frac{1}{2} \left ( \Vert a \Vert_\infty + \sqrt{ \Vert a \Vert_\infty^2 + 4\sigma^2} \right ) \int_{\Omega} \phi_1(x) u_0(x) \ dx > 0.$$
%The proof of this corollary follows from adapting the above argument to the methods of \cite{KS}, in which a similar condition is imposed on the constant damping coefficient.
%\end{remark}

Next, we present the analogue of Theorem \ref{T1} for the first-order evolution problem with  $x \in \Omega$ and $t > 0$:
\begin{equation} \label{paratwo}
\left. \begin{gathered}
 b(t) \partial_t u + Lu = f(x,u)\\
u(0,x) = \vp(x) + u_0(x).\\
\end{gathered} \right \}
\end{equation}
Unlike the second-order case our main theorem proves instability in the $L^\infty$ norm, rather than the $L^2$ norm.  In either situation, these norms are natural to the respective existence theory. In short, we find instability for $b \in L^1_{loc}$ and instability by either exponential growth or blow-up for $b \in L^\infty$. The precise statement of the theorem follows.

\begin{theorem}\label{T2} %(First-order evolution equation)

Assume the hypotheses of Theorem \ref{T1} and $b(t) > 0$ for all $t \in [0,\infty)$. 
Let $T \in (0,\infty]$ and the initial data $u_0 \in L^\infty(\Omega)$ satisfying
$$\int_{\Omega} \phi_1(x) u_0(x) \ dx > 0.$$ be given.  If $u$ is a solution of (\ref{paratwo}) on $[0, T)$ such that $u-\vp$ is $C^1$ and bounded for each $t$. Then, we have the following:
\renewcommand{\labelenumi}{\alph{enumi}.}
\begin{enumerate}
\item There is $C_0 > 0$ such that for all $t \in (0,T)$, $$\Vert u(t) - \vp \Vert_\infty > C_0\int_{\Omega} \phi_1(x) u_0(x) \ dx > 0.$$ 
\item If $b\in L^{\infty}(0,T)$, then there exist $T_0, C_1, \mu > 0$ such that for all $t \in (T_0,T)$ $$\Vert u(t) - \vp \Vert_\infty \geq C_1e^{\mu t}.$$
\item If $b\in L^{\infty}(0,T)$ and the assumptions of Theorem \ref{T1a} hold, then there exists $T^* < \infty$ such that $$\lim_{t \uparrow T^*} \Vert u(t) - \vp \Vert_\infty = +\infty.$$
\end{enumerate}
\renewcommand{\labelenumi}{\arabic{enumi}.}
\end{theorem}

\begin{proof}
As in the proof of Theorem \ref{T1}, we define
$$W(t) = \int_{\Omega} \phi_1(x) w(t,x) \ dx, \quad w(t,x) = u(t,x) - \vp(x),$$ 
for which we have the estimate
$$|W(t)|\leq \|\phi_1\|_{L^1(\Omega)}\cdot \|w\|_{L^{\infty}(\Omega)}.$$
So, as before, the behavior of $W$ will determine the growth of $w$, and hence the instability result. 
We apply the same arguments as those of Theorem \ref{T1}, and by imposing the convexity assumption on $f$ we see that $W$ must satisfy the analogue of (\ref{W}), namely
\begin{equation}\label{WP}
b(t)W'(t)-\sigma^2 W(t) \geq 0.
\end{equation}
Applying Lemma \ref{L5} gives the validity of the first two conclusions of the theorem.

To obtain the last part of the theorem we use the same estimates that yielded (\ref{HP}). The first-order counterpart is now given by the inequality
$$b(t)W'(t) \geq C W(t)^p.$$
By Lemma \ref{L5}(i) we have that $W$ is positive, hence we can directly apply Lemma \ref{L6} to obtain the blow up of $W$ in finite time.
\end{proof}
\noindent In the next section, we prove the exponential growth and blow-up lemmas for functions satisfying the ordinary differential inequalities that are utilized in our main theorems.

\section{Lemmas}

In this section, the lemmas used in the proof of Theorems \ref{T1} -- \ref{T2} are stated and proved. %We note that our methods of proof are quite different from those of previous studies, such as \cite{KS}. 
First, we consider the differential inequalities which arise from our study of the second-order evolution problem (\ref{two}), namely
\begin{equation}
\label{ODE}
a(t)Y''(t) + b(t)Y'(t) - cY(t) \geq 0
\end{equation}
and
\begin{equation}
\label{ODE2}
A(t)Y''(t) + BY'(t) - CY(t)^p \geq 0
\end{equation}
for $p \in (1,\infty)$.

\begin{lemma}
\label{L1}
Let $T \in (0,\infty]$, $a \in C^1(0,T)$ with $a(t) \geq a_0>0$ for all $t \in [0,T)$, $b \in L^1_{loc}(0,T)$, and $c > 0$ be given. Suppose $Y \in C^1$ satisfies (\ref{ODE})
on $[0,T)$ in the sense of distributions with given initial conditions $Y(0), Y'(0) > 0$. Then, $Y'(t) > 0$ for all $t \in [0,T)$.
\end{lemma}

\begin{proof}
To prove the first lemma, we shall take advantage of the regularity of $Y(t)$ and utilize a continuity argument.  Let $\bar{T}<T$ be given so that $\bar{T}<\infty$. Define $$T_1 = \sup \{t \in (0,\bar{T}):  Y'(s) > 0, \mathrm{\  for \ all \ } s \in [0,t] \}.$$ We have $T_1 \leq \bar{T}<T$; moreover, $T_1 > 0$ since $Y'(0) > 0$ by hypothesis. Hence, by continuity there exists $T_0 \in (0,T_1)$ such that 
$$Y'(s)>0 \text{ for all } s\in [0,T_0].$$
Since $Y$ satisfies (\ref{ODE}) on $(0,T)$ in the sense of distributions, it will also satisfy it on $(T_0, T_1)$, hence, after dividing by $a(t)$ we have
$$\int_{T_0}^{T_1} Y'(s) \left ( -\theta'(s) + \frac{b(s)}{a(s)} \theta(s) \right ) ds \geq \int_{T_0}^{T_1} \frac{c}{a(s)}Y(s) \theta(s) \ ds$$ for every nonnegative test function $\theta$.
Since $Y'(s) > 0$ on $[T_0,T_1)$, we find $Y(s) > Y(0) > 0$  on $[T_0, T_1)$ and therefore by using the positivity of $a$ and $c$ we have
$$\int_{T_0}^{T_1} Y'(s) \left ( -\theta'(s) + \frac{b(s)}{a(s)} \theta(s) \right ) \ ds >  0.$$ 
Let $I(s) = \int_{T_0}^s \frac{b(\tau)}{a(\tau)} \ d\tau$ and notice that $I$ remains finite on $[T_0,T_1]$ since $a(s)\geq a_0$ and $b \in L^1_{loc}(0,\infty)$.  Then, rewrite the inequality as $$-\int_{T_0}^{T_1} Y'(s) e^{I(s)}  \frac{d}{ds} \left [ e^{-I(s)} \theta(s) \right ] \ ds >  0.$$ Taking $\theta$ to be an approximation of $e^{I(s)} \chi_{(T_0,T_1)}(s)$, where $\chi_{(\alpha,\beta)}$ represents the characteristic function on of an interval $(\alpha,\beta)$, we deduce $$Y'(T_1) e^{I(T_1)} - Y'(T_0) >  0.$$
Hence, we have $$Y'(T_1) > Y'(T_0)e^{I(T_1)} \geq 0. $$ By the continuity of $Y'(t)$ it must be the case that $T_1 = \bar{T}$, so $Y'(t)>0$ for all $t \in [0,\bar{T})$. With $\bar{T}<T$ arbitrary we have that $Y'(t)>0$ on $ [0,T)$, which is the desired conclusion.
\end{proof}

Next, we prove the exponential growth of solutions under similar conditions, with the additional assumption of $a \in L^\infty$ included.

\begin{lemma}
\label{L2var}
Let $T \in (0,\infty]$, $a(t)\in C^1(0,T)$ satisfying
$$a(t)\geq a_0 >0, \,\,\, a'(t)>0 $$ 
for some $a_0>0 $ for all $t \in (0,T)$, and let $b \in L^\infty(0,T)$, and $c > 0$ be given. Suppose $Y \in C^1$ satisfies (\ref{ODE})
on $[0,T)$ in the sense of distributions with given initial conditions $Y(0) > 0$ and $Y'(0) \geq \left ( \sqrt{\frac{c}{a(0)}+B^2} - B \right ) Y(0)$. Then, there exist $C> 0$ and an increasing function $\mu>0$ such that
$$ Y(t) \geq C e^{\mu(t)}$$ for all $t \in [0,T)$. The function $\mu$ is given by $\mu(t)=\int_0^t (-B+\sqrt{\frac{c}{a(s)}+B^2})ds$, with $B:=\frac{ \Vert b \Vert_\infty}{a_0}.$
\end{lemma}

\begin{proof} We begin by using the positivity of $Y$ and $Y'$, and the lower bound on $a$ to obtain the inequality
$$Y'' +\frac{ \Vert b \Vert_\infty}{a_0} Y' - \frac{c}{a(t)}Y \geq Y'' + \frac{b(t)}{a(t)} Y' - \frac{c}{a(t)} Y \geq 0.$$ 
Next, we multiply this inequality by $e^{Bt}$ where $B:=\frac{ \Vert b \Vert_\infty}{a_0}$.
Denote $f(t):=e^{Bt}Y(t)$, so $f'(t)=e^{Bt}Y'(t)+Be^{Bt}Y(t)$ and $f''(t)=e^{Bt}Y''(t)+2Be^{Bt}Y'(t)+B^2e^{Bt}Y(t)$. In the new variable, the inequality simplifies to
\[
f''-Bf'-\frac{c}{a(t)} f \geq 0.
\]
Since $Y'>0$ as a consequence of Lemma \ref{L1} we find $Bf'\geq B^2 f$, so this yields
\[
f''\geq \left(\frac{c}{a(t)}+B^2\right)f.
\]
After multiplying by $f'\geq 0 $ we get
\[
\frac{d}{dt}[f'(t)^2]-\left(\frac{c}{a(t)}+B^2\right)\frac{d}{dt}[f(t)^2]\geq 0
\]
which after integrating on $(0,t)$ yields
\begin{multline*}
[f'(t)]^2-\left(\frac{c}{a(t)}+B^2\right)[f(t)]^2+\int_0^t \frac{d}{ds}\left(\frac{c}{a(s)}+B^2\right) [f(s)]^2 ds\\ \geq [f'(0)]^2-\left(\frac{c}{a(0)}+B^2\right)[f(0)]^2.
\end{multline*}
Finally, using the assumptions we have $a'(t)>0$ and the right side of the inequality is nonnegative. It follows that
$$ f'(t) \geq \sqrt{\frac{c}{a(t)} + B^2} \  f(t)$$ and thus
\[
f(t)\geq f(0) e^{\int_0^t\sqrt{\frac{c}{a(s)}+B^2}ds}. 
\] 
We end up with the following growth rate for $Y$:
\[
Y(t)\geq Y(0)e^{\int_0^t (-B+\sqrt{\frac{c}{a(s)}+B^2})ds}.
\]
\end{proof}

For the case when $a\in L^{\infty}(0,\infty)$ we have the following result, for which we are able to provide a more direct proof.
\begin{lemma}
\label{L2}
Let $T \in (0,\infty]$, $a(t)\in C^1(0,T)$ satisfying the growth condition 
$$0<a_0\leq a(t) \leq a_1  $$ 
for some $a_0, a_1 >0 $ and  all $t \in [0,T)$, $b \in L^\infty(0,T)$, and $c > 0$ be given. Suppose $Y \in C^1$ satisfies (\ref{ODE})
on $[0,T)$ in the sense of distributions with given initial conditions $Y(0), Y'(0) > 0$. Then, there exist $C, \mu > 0$ such that
$$ Y(t) \geq C e^{\mu t}$$ for all $t \in [0,T)$.
\end{lemma}

\begin{proof}
Since $b \in L^\infty(0,T)$, and hence $b \in L_{loc}^1(0,T)$, we apply Lemma \ref{L1} to conclude $Y' > 0$ on $[0,T)$. Using this and the inequality (\ref{ODE}) we find 
$$Y'' +\frac{ \Vert b \Vert_\infty}{a_0} Y' - \frac{c}{a_1}Y \geq Y'' + \frac{b(t)}{a(t)} Y' - \frac{c}{a(t)} Y \geq 0.$$  
Put $\beta := \displaystyle\frac{\Vert b \Vert_\infty}{a_0}$, $\gamma:=\displaystyle\frac{c}{a_1}$ and let $\lambda_{\pm}$ be the roots of the characteristic equation $$\lambda^2 + \beta \lambda - \gamma = 0.$$ Since $\gamma > 0$, we may assume $\lambda_+ > 0$ and $\lambda_- < 0$. Thus, set $Z(t) = Y'(t) - \lambda_- Y(t)$ and compute
$$Z' - \lambda_+ Z = Y'' + \beta Y' -\gamma Y \geq 0$$
which implies $Z(t) \geq Z(0) e^{\lambda_+ t}$ and by the definition of $Z$, $$Y'(t) - \lambda_- Y(t) \geq Z(0) e^{\lambda_+ t}.$$ Finally, this generates a lower bound for $Y$, namely $$ Y(t) \geq  \frac{Z(0)}{\lambda_+ - \lambda_-} (e^{\lambda_+ t} - e^{\lambda_- t} ) + Y(0) e^{\lambda_- t}.$$ Since $\lambda_- < 0$, we have $Z(0) = Y'(0) - \lambda_- Y(0) > 0$ by assumption. Hence, we see that 
$$Y(t) \sim \frac{Z(0)}{\lambda_+ - \lambda_-} e^{\lambda_+ t}$$ as $t$ grows large and because $\lambda_+ > 0$, the exponential growth of $Y$ follows.
\end{proof}

Now we prove that a similar exponential growth result for $b \in L^1(0,\infty)$ also holds, but using different methods.

\begin{lemma}
\label{L3} Let $T \in (0,\infty]$, $a\in C^1(0,T)$ with $0<a_0\leq a(t) \leq a_1$ for all $t \in [0,T)$, $b \in L^1(0,\infty)$, and $c > 0$ be given. Suppose $Y \in C^1$ satisfies (\ref{ODE})
on $[0,T)$ in the sense of distributions with given initial conditions $Y(0), Y'(0) > 0$. Then, there exists $C, \mu > 0$ such that
$$ Y(t) \geq Ce^{\mu t}$$ for all $t \in [0,T)$.
\end{lemma}

\begin{proof}
Let the functions $a, b$ and the constant $c$ be given satisfying the conditions of the lemma.  Choose $$ \mu = \biggl ( \frac{21c}{40a_1} e^{-\Vert \frac{ b }{a}\Vert_{_1}} \biggr )^{1/2} > 0.$$Define $$T_0 = \sup \left \{ t \in (0,T) :  Y(s) \geq \frac{1}{2}  Y(0) e^{\mu t},  \quad \forall s \in [0,t] \right \}$$ and notice that $T_0 > 0$.  Thus, for any $t \in [0,T_0]$, we can proceed as in the proof of Lemma \ref{L1} to express (\ref{ODE}) as
$$-\int_0^t Y'(s) e^{I(s)} \frac{d}{ds} \left [ e^{-I(s)} \theta(s) \right ] \ ds \geq \int_0^t \frac{c}{a(s)}Y(s) \theta(s) \ ds$$ where $I(s) = \int_0^s \frac{b(\tau)}{a(\tau)} \ d\tau$.  As before, we take $\theta$ to be an approximation of the integrating factor multiplied by the characteristic function on $(0,t)$ and the inequality becomes
$$ e^{I(t)} Y'(t) - Y'(0) \geq  \int_0^t \frac{c}{a(s)} e^{I(s)} Y(s) \ ds.$$
Since $t \in [0,T_0]$ this implies
$$  Y'(t) \geq  Y'(0) e^{-I(t)} + \frac{c}{2a_1}  Y(0) \int_0^t e^{I(s) - I(t)} e^{\mu s} \ ds.$$ 
Since $\frac{b}{a}$ is integrable, we find
$$  Y'(t) \geq e^{-\Vert \frac{b}{a} \Vert_1} \biggl ( Y'(0) + \frac{ c}{2a_1\mu}  Y(0) (e^{\mu t} -1 ) \biggr )$$ and after an integration
$$  Y(t) \geq  Y(0)  + e^{-\Vert \frac{b}{ a} \Vert_1} \biggl (  Y'(0)t + \frac{ c}{2a_1\mu}  Y(0) \left [ \frac{1}{\mu}(e^{\mu t} -1 ) - t \right ] \biggr ).$$
Reordering terms, this becomes
\begin{equation}
\label{five}  Y(t) \geq  Y(0) \biggl ( 1 - e^{-\Vert\frac{b}{ a} \Vert_1} \frac{ c}{2a_1\mu^2} \biggr ) + e^{-\Vert \frac{b}{a} \Vert_1}  Y'(0) t +  e^{-\Vert \frac{b}{a} \Vert_1}  Y(0) \frac{c}{2a_1\mu^2}  \biggl (e^{\mu t} - \mu t \biggr ).
\end{equation}
By definition, $\mu^2 \geq \frac{c}{a_1}e^{-\Vert \frac{b}{a} \Vert_1}  $, thus the first term on the right side of (\ref{five}) is nonnegative.  Similarly, the second term is nonnegative by hypothesis.  Finally, call the last term $L$.  Since $$e^x > \frac{5}{2} x$$ for all $x \in \mathbb{R}$ and $ Y(0) \geq 0$ we find
\begin{eqnarray*}
L & = &   \frac{3}{5}e^{-\Vert \frac{b}{a} \Vert_1}  Y(0) \frac{c}{2a_1\mu^2} e^{\mu t} + \frac{2}{5}e^{-\Vert \frac{b}{a} \Vert_1}  Y(0) \frac{c}{2a_1\mu^2}  \biggl (e^{\mu t} - \frac{5}{2}\mu t \biggr )\\
& > & \frac{3}{5}e^{-\Vert \frac{b}{a} \Vert_1}  Y(0) \frac{ c}{2a_1\mu^2} e^{\mu t}\\
& = & \frac{3}{5} \cdot \frac{40}{21} \cdot \frac{1}{2}  Y(0) e^{\mu t}\\
& = & \frac{8}{7} \left ( \frac{1}{2}  Y(0) e^{\mu t} \right ).
\end{eqnarray*}
Thus, (\ref{five}) yields the lower bound
\begin{equation}
\label{six}  Y(t) \geq \frac{8}{7} \left ( \frac{1}{2}  Y(0) e^{\mu t} \right )
\end{equation} for $t \in [0,T_0]$.  So, we find $ Y(T_0) \geq \frac{4}{7}  Y(0) e^{\mu T_0}$ and this implies $T_0 = T$.  Hence, the lower bound (\ref{six}) is valid for all $t \in [0,T) $ and it follows that $ Y(t)$ must grow exponentially.
\end{proof}

We state the next lemma regarding the blow-up of solutions to justify the proof of Theorem \ref{T1a}. A nice proof of this blow-up result is given in Proposition 3.1 in \cite{TY}.
\begin{lemma}
\label{L4}
Let $a\in C^1(0,T)$ with $0< a_0 \leq a(t) \leq a_1$ for all $t \in [0,T),\,$$B, C > 0$, $r\in [0,1),$ and $p > 1$ be given. Let $Y \in C^1$ be a solution of (\ref{ODE2}) 
\[
Y''(t) + BY'(t) - C (1+t)^{-r}Y(t)^p \geq 0
\]
on the interval $[0,T)$ in the sense of distributions with given initial conditions $Y(0), Y'(0) > 0$. Then $T < \infty$.
\end{lemma}

Finally, we turn to the corresponding lemmas for the first-order equation and consider the differential inequalities
\begin{equation}\label{ODEH}
b(t)Y'(t)- c Y(t) \geq 0
\end{equation}
and
\begin{equation}\label{ODEH2}
b(t)Y'(t)- c Y(t)^p \geq 0
\end{equation}
for $p \in (1,\infty)$, which arise in the proof of Theorem \ref{T2}.

\begin{lemma}
\label{L5}
Let $T \in (0,\infty]$, $b \in L^1_{loc}(0,T)$ be positive, and $c > 0$ be given. Suppose $Y \in C^1$ satisfies (\ref{ODEH})
on $[0,T)$ with given initial condition $Y(0) > 0$. Then
\begin{enumerate}
\item[(i)] $Y(t) > \displaystyle Y(0)$, for all $t\in [0, T)$;
\item[(ii)] Moreover, if $a \in L^\infty(0,T)$ there exist $C, \mu >0$ such that 
\[
Y(t)\geq C e^{\mu t}
\]
 for all $t\in [0, T)$.
\end{enumerate}
\end{lemma}
\begin{proof}
We proceed by using a continuity argument yet again and define
$$T_0 = \sup \{t \in (0,T):  Y(s) >\frac{1}{2}Y(0),  \quad \forall  s \in [0,t] \}.$$ 
Note that on $[0, T_0)$ we have by (\ref{ODEH}) $$ b(t) Y'(t) \geq c Y(t) > \frac{1}{2} c Y(0) > 0.$$ Hence, the positivity of $b$ implies $Y'(t) > 0$ as well. This yields $Y(t) > Y(0)$ for all $t \in (0,T_0)$, whence $T_0=T$, and the first conclusion of the lemma holds. For the second part of the lemma, we use the $L^{\infty}$ bound on $b$ and the positivity of $Y$ to obtain from (\ref{ODEH}) the inequality
$$ \Vert b\Vert_\infty Y'(t) \geq b(t) Y'(t) \geq c Y(t)$$ and thus
$$Y'(t) \geq c\Vert b \Vert_\infty^{-1} Y(t).$$
As in the proof of Lemma \ref{L2}, the exponential growth of $Y$ on $[0,T)$ follows as an immediate consequence.
\end{proof}

Our final lemma of the section establishes the blow-up result for (\ref{ODEH2}) under the assumption of a bounded damping coefficient.

\begin{lemma}\label{L6} 
Let $b\in L^{\infty}(0, T)$ be positive with $c > 0$, and $p > 1$ given. Let $Y \in C^1$ be a solution of (\ref{ODEH2}) on the interval $[0,T)$ with given initial condition $Y(0) > 0$. Then $T < \infty$.
\end{lemma}
\begin{proof}
Assume $T = \infty$.  As in the previous lemma, a continuity argument easily establishes $Y(t) > Y(0)$ and $Y'(t) > 0$ for $t \in [0,T)$.
Using this (\ref{ODEH2}) implies 
\[
\|b\|_{\infty}Y'(t)\geq b(t)Y'(t)\geq cY^p(t).
\]
The inequality can be rearranged in separable form 
\[
\frac{Y'(t)}{Y^p(t)} \geq \frac{c}{\|b\|_{\infty}}
\]
which after integration yields
\[
Y(t)\geq \left[ Y^{1-p}(0) - (p-1) \frac{c}{\|b\|_{\infty}}t\right]^{1/(1-p)}.
\]
Since there exists a time $T_\infty<\infty$ such that  $$Y^{1-p}(0) = (p-1)\frac{c}{\|b\|_{\infty}}T_\infty ,$$ we conclude that $Y$ blows up as $t \uparrow T_\infty$.

\end{proof}

%\begin{remark}
%Lemma \ref{L6} holds under a different assumption that the coefficient $a$ satisfies
%  $$\frac{1}{a}\in L^1_{loc}(0,T), \quad \int_0^{\infty} \frac{1}{a(s)}ds=\infty.$$ 
%  In fact only the milder condition 
%  \[
%\int_0^{\infty} \frac{1}{a(s)}ds > \frac{1}{C(p-1)Y^{p-1}(0)}
%\]
%is necessary. This last assumption, however, would be more difficult to verify in the context of Theorem \ref{T2}, since it would involve computing both the eigenfunction $\phi_1$ and the eigenvalue $-\sigma^2$ on which $C$ depends. We prefer the condition that $a\in L^{\infty}$ for part (b) of Theorem \ref{T2}, since this coincides with the assumption needed for part (a) and it is also very easy to verify in practice.   
%\end{remark}
%%%%%%%%%%%%%%%%%%%%

\section{Applications of Theorems \ref{T1}-\ref{T2}}

Our results are immediately applicable to a few well-known and well-studied problems, including steady states described in the aforementioned papers \cite{CL, GNW, KS, SZ}.  Namely, for steady state solutions corresponding to the following equations (and their parabolic counterparts)
$$a(t)\partial_{tt} u+b(t)\partial_t u -\Delta u=|u|^p, \quad x \in \Omega, \, t > 0$$
$$a(t)\partial_{tt}u + b(t) \partial_t u - \Delta u +V(x) u = f(u), \qquad x \in \Omega, t > 0$$
the instability results of our main theorem via exponential growth, and in some cases blow-up, are valid under reasonable assumptions on the associated nonlinearity. To our knowledge, the instability results of this section are new for both the hyperbolic and parabolic equations with unbounded time-dependent (and possibly sign-changing) coefficients. The proofs rely mostly on the analysis of the steady state solutions that was performed in \cite{KS}, from which we include only the essential findings.

\subsection{Laplacian with power nonlinearity}
Our first application considers solutions of the steady state equation
\begin{equation}\label{wep}
-\Delta  \vp = \vert \vp \vert^p, \quad x\in \Omega
\end{equation}
with $n>2$. For $\Omega = \bfR^n$, Li \cite{Li} proved the existence of an uncountable family of positive solutions $\vp \in C^2(\bfR^n)$ to (\ref{wep}) for $p > \frac{n+1}{n-2}$.  Shortly afterward, Chen and Li \cite{CL} explicitly determined the form of all positive, $C^2$ solutions for $p = \frac{n+1}{n-2}$ and showed that for $p < \frac{n+2}{n-2}$ or $n=1,2$ no such solutions of (\ref{wep}) can exist.  More recently, in \cite{KS} it was shown that for $p \geq \frac{n+2}{n-2}$, the linearized operator $\mathcal{L} = -\Delta - p\vp^{p-1}$  has a negative eigenvalue if
$$\left(\frac{n-2}{2}\right)^2< \frac{2p}{p-1}\left(n-2-\frac{p}{p-1}\right).$$
On the other hand, if we have $ p> \frac{n+2}{n-2}$, but 
$$\left(\frac{n-2}{2}\right)^2 \geq \frac{2p}{p-1}\left(n-2-\frac{p}{p-1}\right)$$
then the linearized operator has no negative spectrum.  It was then proved that steady states are nonlinearly unstable in the former case and stable in the latter.  We can now extend these results to prove the dichotomy given by the following theorem:
\begin{theorem}\label{T43}
Let $n>2$ and $p\geq \frac{n+2}{n-2}$ be given, and let $\vp \in C^2(\Omega)$ be a positive solution of (\ref{wep}). Denote
$$p_c=\begin{cases}
\infty, & n\leq 10,\\
\frac{n^2-8n+4+8\sqrt{n-1}}{(n-2)(n-10)}, & n>10,
\end{cases}$$
for convenience. Then, $p_c>\frac{n+2}{n-2}$ and the following dichotomy holds.

If $p<p_c$, then $\vp$ is a nonlinearly unstable solution of the equations
$$a(t)\partial_{tt} u+b(t)\partial_t u -\Delta u=|u|^p, \quad x \in \Omega, \, t > 0$$
and
$$b(t)\partial_t u-\Delta u=|u|^p, \quad x\in \Omega, \, t > 0$$
where $b(t)$ satisfies the assumptions given in Theorems \ref{T1} -- \ref{T2}, and nonlinear instability occurs as given by the aforementioned theorems.

Contrastingly, if $p\geq p_c$, then $\vp$ is a stable solution of the linearized equation, i.e. the linearized operator has no negative spectrum.
\end{theorem}
To our knowledge, the instability results of Theorem \ref{T43} are new for both the parabolic and hyperbolic equations with time-dependent damping.  For the nonlinear heat equation, the stability of steady states was studied earlier by Gui, Ni, and Wang \cite{GNW}.  They arrived at a theorem very similar to the dichotomy above, though only the linear stability and instability were proved.  As the details of the proof of the theorem are closely related to those presented by Strauss and Karageorgis, we omit it and refer the reader to \cite{KS} for more information.

If $\Omega$ is a bounded subset of $\bfR^n$, then less is generally known.  As mentioned in the introduction, the problem with a concave power nonlinearity has been studied in \cite{FM, FM1}.  The authors proved the stability of steady states to such problems. Using our main theorems, we can arrive at complementary results for steady states satisfying
\begin{equation}\label{wep2}
-\Delta  \vp = -\vert \vp \vert^p, \quad x\in \Omega.
\end{equation}

\subsection{Laplacian with a potential and convex nonlinearity}
Next, we consider general conditions upon which the solutions of the equation
\begin{equation}
\label{steadypotential}
-\Delta \vp + V(x) \vp =f(\vp), \qquad x \in \bfR^n
\end{equation}
are nonlinearly unstable solutions of 
\begin{equation}
\label{nh}
a(t)\partial_{tt}u + b(t) \partial_t u - \Delta u +V(x) u = f(u), \qquad x \in \Omega, t > 0
\end{equation}
and
\begin{equation}
\label{np}
b(t) \partial_t u - \Delta u +V(x) u = f(u), \qquad x \in \Omega, t > 0.
\end{equation}
by utilizing the main theorems of Section $2$.  We will assume for these problems that (\ref{steadypotential}) has a positive solution $\vp \in C^2 \cap H^1$ that satisfies $\vp(x) \to 0$ as $\vert x \vert \to \infty$.  For such steady states, we have the following result:
\begin{theorem}
\label{convexapp}
Assume that the potential $V: \Omega \to \bfR$ is continuous, bounded, and satisfies $\sigma_{ess} (-\Delta + V) \subset [0,\infty)$.  In addition, assume $f \in C^1(\bfR)$ is convex with $f(0) = f'(0) = 0$ and not identically zero. Then, $\vp$ is a nonlinearly unstable solution of (\ref{nh}) and (\ref{np}). More precisely, the conclusions of Theorems \ref{T1} -- \ref{T2} are valid.
\end{theorem}

\begin{proof}
The proof of this result relies on checking the necessary assumptions for the steady states and the adjoint of the linearized operator. Consider the self-adjoint linearized operator $$\mathcal{L} = -\Delta + V - f'(\vp).$$ Since $f(0) = 0$ and $f$ is convex, we notice that for any solution of (\ref{steadypotential}) 
\begin{eqnarray*}
\mathcal{L} \vp & = & -\Delta\vp + V(x)\vp - f'(\vp)\vp\\
& = & f(\vp) - f'(\vp)\vp\\
& = & f(\vp) - f(0) - f'(\vp)\vp\\
& \leq & 0.
\end{eqnarray*}
Now, since $f'(0) = 0$ and $f$ is not identically zero, it follows that $\mathcal{L}\vp \neq 0$.  Thus, we have $\mathcal{L}\vp < 0 < \vp$, whence $\langle \mathcal{L}\vp,\vp \rangle < 0$ and the first eigenvalue of $\mathcal{L}$ must be negative.  Finally, it has previously been determined \cite{LL, Simon} that the corresponding first eigenfunction $\phi_1$ belongs to $L^1(\Omega) \cap L^2(\Omega)$ and is non-negative. Hence, all assumptions are satisfied and the results follow by an application of Theorems \ref{T1} -- \ref{T2}.
\end{proof}

\begin{remark}
In the special case that $f(u) = \vert u \vert^p$ with $n > 2$ and $ p \in \left (1,\frac{n+2}{n-2} \right )$, Souplet and Zhang \cite{SZ} have studied the stability problem for radial potentials that satisfy $$C_1 (1 + \vert x \vert)^{-q} \leq V(x) \leq C_2$$  with $0 \leq q \leq \frac{2(n-1)(p-1)}{p+3}$. They previously showed the existence of a positive steady state $\vp \in C^2$ that decays exponentially fast as $\vert x \vert \to \infty$.  Theorem \ref{convexapp} applies to this particular case, in addition to a much wider class of potentials. Moreover, we notice that the assumptions are trivially satisfied since the steady state $\vp$ is bounded.  Thus, for the second-order problem (\ref{nh}), the $L^2$ norm $\|u(t)-\vp\|_{2}$ blows up in finite time, and for the parabolic problem (\ref{np}), the $L^\infty$-norm $\|u(t)-\vp\|_{\infty}$, blows up in finite time, as in our main results.
\end{remark}

\subsection{Exponential nonlinearity in 2D:}

Using the theorems of Section $2$, one may also transfer instability from the linearized systems to the nonlinear problems
$$a(t)\partial_{tt} u + b(t)\partial_t u -\Delta  u =e^{u}, \quad x\in \bfR^2, \, t > 0$$
and 
$$b(t)\partial_t u-\Delta  u=e^{u}, \quad x\in \bfR^2, \, t > 0.$$
While this is true for the general instability and exponential growth results, we note that our blow-up results will not hold for these equations as the exponential nonlinearity does not satisfy the necessary assumptions. Smooth solutions of the time-independent problems, which satisfy
\begin{equation}
\label{steadyexp}
-\Delta \vp = e^{\vp},
\end{equation}
have already been studied. To be more precise, the $C^2$ steady states of these systems were completely categorized by Chen and Li \cite{CL}.  In this paper, the authors were able to prove the existence of an infinite family of such solutions, each of which must take the form
$$\vp(x) = \log \left ( \frac{32\lambda^2}{ \left ( 4 + \lambda^2 \vert x - y \vert^2 \right )^2} \right )$$
for some $\lambda>0$ and $y \in \bfR^2$.  With this representation and the explicit form of the nonlinearity, we can determine that our main theorems are applicable by checking that their assumptions are valid for these equations. Upon doing so, we may prove the following instability result, which extends known results to the case of variable coefficients.
\begin{theorem}
Assuming $e^u \in L^1(\bfR^2)$ and the damping term $b(t)$ is in the appropriate $L^p$ space (as stipulated in Theorems \ref{T1} and \ref{T2}), the steady state $\vp$ satisfying (\ref{steadyexp}) is a nonlinearly unstable solution of the equations
$$a(t)\partial_{tt}u + b(t) \partial_t u - \Delta u = e^u, \quad x\in \bfR^2, \, t > 0$$
and
$$b(t)\partial_t u-\Delta  u=e^{ u}, \quad x\in \bfR^2, \, t > 0$$
on $[0,T)$. More specifically, conclusions (a), (b), and (d) of Theorem \ref{T1} and conclusions (a) and (b) of Theorem \ref{T2} are valid.
\end{theorem}

\begin{proof}
As in the previous proof, all that is needed is to verify the necessary assumptions. Note that the nonlinearity $f(y)=e^y \in C^\infty$ is convex. Additionally, it is known \cite{KS} that the linearized operator $$v \to \mathcal{L}v = -\Delta v - e^{\vp(x)} v$$ has a negative first eigenvalue and positive first eigenfunction.  This is done by proving that the associated energy functional $$E[v] =  \int_{\bfR^2} \left ( \vert \nabla v(x) \vert^2 - e^{\vp(x)} \vert v(x) \vert^2 \right ) \ dx$$ achieves a negative value for a specific choice of $v$.  Indeed, selecting the function $\xi(x) =  \left (4 + \lambda^2 \vert x - y \vert^2 \right )^{-2}$, one finds $E[\xi] = -\frac{1}{320} \pi$, and this implies the existence of both a negative eigenvalue and corresponding positive eigenfunction. Hence, all of the assumptions are satisfied, and Theorems \ref{T1} and \ref{T2} yield the result. 
\end{proof}

In summary, we would like to remark that the main results are also applicable to equations with higher-order spatial operators since the order of $L$ was unnecessary in our proofs. Examples of such problems include those which involve the biharmonic or polyharmonic operators, such as
\begin{equation}
\label{biharmonic}
\left \{ \begin{gathered}
\Delta^2 v = f(v), \qquad x \in \Omega\\
v = \frac{\partial u}{\partial n} = 0, \qquad x \in \partial\Omega.
\end{gathered} \right. 
\end{equation}
In general, little information is currently known regarding steady state solutions to such problems. Hence, it is unclear as to whether our general assumptions are satisfied; for instance the existence of a positive principal eigenfunction or an associated, linearized adjoint operator possessing negative eigenvalues. In fact, the case of general domains with Dirichlet boundary conditions is essentially unexplored, often due to the lack of a comparison principle. One promising candidate is the Gelfand problem (\ref{biharmonic}) with $f(v) = e^v$, whose linearized operator is known to possess negative eigenvalues \cite{BFFG}, at least for $\Omega = \Rn$.  Unfortunately, the corresponding eigenfunctions may change sign, and hence our theorems do not apply. With advances in the study of such problems, we believe our main results will also give rise to new applications for these equations in the future.

\end{document}